\documentclass[11pt]{article}

\usepackage[margin=1in]{geometry}

\usepackage{amsmath,amssymb,amsthm}
\usepackage{hyperref}
\usepackage{graphicx}
\numberwithin{equation}{section} % Eqns numbered per section
\allowdisplaybreaks              % Allow multi-line displays to break across pages

\usepackage{xcolor}
\definecolor{rptu1}{RGB}{4,44,88}
\definecolor{rptu2}{RGB}{106,178,231}
\definecolor{rptublaugrau}{RGB}{80,114,137}
\definecolor{rptugruengrau}{RGB}{119,182,186}
\definecolor{rptudunkelblau}{RGB}{4,44,88}
\definecolor{rptuhellblau}{RGB}{106,178,231}
\definecolor{rptudunkelgruen}{RGB}{0,107,107}
\definecolor{rptuhellgruen}{RGB}{38,208,124}
\definecolor{rptuviolett}{RGB}{76,53,117}
\definecolor{rptupink}{RGB}{209,56,150}
\definecolor{rpturot}{RGB}{227,27,76}
\definecolor{rptuorange}{RGB}{255,162,82}
\definecolor{rptuschwarz}{RGB}{0,0,0}
\definecolor{rptuweiss}{RGB}{255,255,255}
\definecolor{darkpurple}{RGB}{75,0,130}

\usepackage{animate}
\usepackage{graphicx}
\usepackage{cancel}

\usepackage{appendix}

\title{Hypocoercive Langevin dynamics on the Lie group $\textbf{SE}(2)$}

\author{
  Martin Grothaus\\
  \textit{Department of Mathematics, RPTU University  Kaiserslautern--Landau}\\
  \texttt{grothaus@rptu.de}\\[1em]
  Andrea V. Hurtado-Quiceno\\
  \textit{Department of Mathematics, RPTU University  Kaiserslautern--Landau}\\
  \texttt{avanessa.hquiceno@edu.rptu.de}
}

\newcommand{\AuthorInfo}[3]{%
  \noindent\textbf{#1}\\#2\\\textit{E-mail address:} \texttt{#3}\vspace{1em}\\
}

\usepackage{enumitem}

\DeclareMathOperator{\A}{\mathcal{A}}

\def\L{\mathcal{L}}
\def\S{\mathcal{S}}

\def\Hi{\mathcal{H}}

\def\P{\mathrm{\Pi}}

\def\Hi{\mathcal{H}}

\def\SE{\textbf{$\mathrm{SE}$}}

\def\P{\mathrm{\Pi}}

\theoremstyle{plain} % this makes the body text italic

\usepackage{cancel}
\newtheorem{theorem}{Theorem}[section]

\newtheorem{proposition}[theorem]{Proposition}

\newtheorem{definition}[theorem]{Definition}

\newtheorem{remark}[theorem]{Remark}
\numberwithin{equation}{section}

\hypersetup{
    colorlinks=false,
    pdfborder={0 0 1},   % forces visible border
    citebordercolor=rptupink
}

\usepackage{parskip}
\makeatletter
\newcommand\ackname{Acknowledgements}
\if@titlepage
  \newenvironment{acknowledgements}{%
      \titlepage
      \null\vfil
      \@beginparpenalty\@lowpenalty
      \begin{center}%
        \bfseries \ackname
        \@endparpenalty\@M
      \end{center}}%
     {\par\vfil\null\endtitlepage}
\else
  \newenvironment{acknowledgements}{%
      \if@twocolumn
        \section*{\abstractname}%
      \else
        \small
        \begin{center}%
          {\bfseries \ackname\vspace{-.5em}\vspace{\z@}}%
        \end{center}%
        \quotation
      \fi}
      {\if@twocolumn\else\endquotation\fi}
\fi
\makeatother

\date{\today}

\begin{document}
\maketitle

\abstract{ \noindent We consider a Langevin-type diffusion on the planar motion group $\SE(2)$, describing the coupled evolution of position and orientation with degenerate noise acting only in the rotational direction. Although hypocoercivity for related models on $\mathbb{R}^2 \times \mathbb{S}^1$ is well understood, our purpose is to present an intrinsic formulation on the Lie group $\SE(2)$, and to highlight the underlying geometric mechanism. By expressing the generator in terms of invariant vector fields and using the natural projection onto the kernel of the symmetric part, we show how an effective macroscopic diffusion on $\mathbb{R}^2$ emerges through averaging over the compact rotation subgroup.}

%-----------
\tableofcontents
\bigskip

%-------------------------

%%%%%%%%%%%%%%%%%%%%%%%%%%%%%%%%%%%%%%%%%%
\section{Introduction}\label{sec:1}

In this article, we study hypocoercive Kolmogorov-type operators on the Lie group $\SE(2)$, the group of orientation-preserving Euclidean motions of the plane. This Lie group carries a natural semidirect product structure, combining translations and rotations, and provides a canonical configuration space for stochastic models describing coupled position-orientation dynamics. Such models arise naturally in kinetic theory, geometric mechanics, and the study of degenerate diffusion processes on manifolds.\\

\noindent It is well known, following H\"ormander's celebrated theorem, that differential operators of the form

\begin{equation*}
\L = \sum_{i=1}^m X_i^2 + X_0,
\end{equation*}

\noindent where the vector fields $\{X_i\}$ generate the full tangent space through their commutators, are hypoelliptic. While hypoellipticity concerns regularity properties of solutions, the focus of the present work is of a different nature. We investigate hypocoercivity, that is, the convergence to equilibrium for degenerate Markov generators whose symmetric part fails to be coercive.\\

\noindent More precisely, we consider operators admitting a decomposition $\L = \S - \A$, where the symmetric part $\S$ acts only on a subset of directions and is therefore non-coercive, while the antisymmetric part $\A$ couples the remaining directions through the commutators. Although $\S$ alone does not control all degrees of freedom, the interaction between $\S$ and $\A$ yields convergence to equilibrium.\\ 

\noindent The hypocoercivity term was developed first by Cedric Villani in 
\cite{Villani2009}. The analysis in the present work is carried out within the 
abstract Hilbert space hypocoercivity framework introduced by Dolbeault, 
Mouhot, and Schmeiser \cite{dolbeault2015hypocoercivity}, and subsequently 
refined to address domain issues in \cite{GrotStil}. Rather 
than focusing on technical refinements of the method, we use this framework as 
a tool to highlight the geometric features of the model on $\SE(2)$.\\

\noindent Ergodicity with explicit rate for related models on product spaces, such 
as $\mathbb{R}^2 \times \mathbb{S}^1$, has been studied in detail in previous works, see e.g. \cite{GrotKlar08}. In these approaches, the geometry of the 
state space is typically encoded through explicit coordinate representations. 
The purpose of the present article is to reformulate these dynamics 
intrinsically on the Lie group $\SE(2)$, and to make explicit the role played 
by invariant vector fields and group structure.\\

\noindent Although the convergence properties are consistent with existing results, the 
Lie group formulation offers a geometric interpretation of hypocoercive 
Langevin dynamics on $\SE(2)$. In particular, the contribution of the 
translation variables to the macroscopic coercivity becomes apparent through 
the orthogonal projection that averages over the rotational component. This 
Lie group perspective clarifies the structural origin of hypocoercivity in 
orientation-position dynamics on $\SE(2)$.

%%%%%%%%%%%%%%%%%%%%%%%%%%%%%%%%%%
\section{Geometry of $\SE(2)$}\label{subsec:se2_geometry_vecfields}

%\textcolor{red}{Check this reference really well.} \\
The special Euclidean $\SE(2)$ group describes the orientation-preserving isometries of the Euclidean space $\mathbb{R}^2$. The special Euclidean group $\SE(2)$,
can be represented as

\begin{equation}
\begin{aligned}
\SE(2) & :=\left\{\left. \left(\begin{array}{cc}
R_{\theta} & \xi \\ 
0_{1 \times 2} & 1
\end{array}\right) \in \mathbb{R}^{3 \times 3} \right\rvert\, \theta \in(0, 2\pi], \xi \in \mathbb{R}^2\right\},
\end{aligned}
\end{equation}

\noindent where $$R_{\theta} :=\left(\begin{array}{cc}
\cos (\theta) & -\sin (\theta) \\
\sin (\theta) & \cos (\theta)
\end{array}\right).$$ 

\noindent As a smooth manifold, $\SE(2)$ is diffeomorphic to the semidirect product $\mathbb{R}^2 \rtimes \mathbb{S}^1$, and we use global coordinates $(\xi,\theta)$ with $\xi=(\xi^1,\xi^2)\in\mathbb{R}^2$ and $\theta\in\mathbb{S}^1$. The Lie algebra $\mathfrak{se}(2)$ is three-dimensional, and can be identified with the tangent space at the identity. In these coordinates, we consider the following invariant vector fields, see e.g. \cite[Sec. 3]{DuitsErik}:

\begin{equation}\label{vec_fields_SE(2)}
\begin{aligned}
X_1 &= \cos(\theta)\,\partial_{\xi^1} + \sin(\theta)\,\partial_{\xi^2}, \quad X_2= \partial_\theta.
\end{aligned}
\end{equation}

\noindent The remaining direction (translation) is generated through the commutator

$$[X_2, X_1] = - \sin(\theta) \partial_{\xi^1} + \cos(\theta) \partial_{\xi^2}=:X_3.$$

%%%%%%%%%%%%%%%%%%%%%%%%%%%%%%%%%%%%5
\section{Hypocoercive Langevin dynamics on $\SE(2)$}\label{sec:generator}

Consider the Hilbert space $\mathcal{H} := L^2(\SE(2),\mu_{\Phi})$, where $ (\xi,  \theta) \in \SE(2)= \mathbb{R}^2 \rtimes \mathbb{S}^1$, endowed with the probability measure

\begin{equation}\label{proba_measure_SE(2)}
d\mu_{\Phi} := \frac{1}{Z} e^{-\Phi(\xi)} d\xi \otimes \frac{1}{(2\pi)} d \theta, %=  \frac{1}{(2\pi Z)} e^{-\Phi(\xi)} d\xi \otimes  d \theta,  
\end{equation}

\noindent where the potential $\Phi: \mathbb{R}^2 \rightarrow \mathbb{R}$, and $Z= \int_{\mathbb{R}^2} e^{-\Phi(\xi)} d\xi$ is the normalization term in such a way that $\mu_{\Phi}(\SE(2) )=1$.

%\vspace{-0.7cm}
\subsection{Generator and operator decomposition}\label{subsec:generator_decomposition}

Let $D := C_c^\infty(\SE(2))$. For $\sigma \in (0, \infty)$, we define operators
$(\S,D)$ and $(\A,D)$ by

\begin{equation}
\begin{aligned}
%\begin{cases}
    \S&= - \frac{\sigma^2}{2} X_{2}^{*} X_2  %=  \frac{\sigma^2}{2} X_{2} X_{2} 
    = \frac{\sigma^2}{2} \Delta_{\theta} ,\\[6pt]
    \A %&= v(\theta) \cdot \nabla_{\xi} - (\nabla_\xi \Phi(\xi) \cdot v_{\perp}(\theta)) \cdot \partial_{\theta}\\
    & %= \cos(\theta) \partial_{\xi^1} + \sin(\theta) \partial_{\xi^2} - (\nabla_\xi \Phi(\xi) \cdot v_{\perp}(\theta)) \cdot \partial_{\theta} 
    = X_1 - (\nabla_\xi \Phi(\xi) \cdot v_{\perp}(\theta)) \cdot X_2,
%\end{cases}
\end{aligned}
\end{equation} where $v(\theta)= (\cos(\theta), \sin(\theta))$, $v_{\perp}(\theta):= \frac{\partial v(\theta)}{\partial \theta}= (-\sin(\theta), \cos(\theta))^{\top}$, and $\nabla_{\xi} = (\partial_{\xi^1}, \partial_{\xi^2})$. Let us note that the operator $\L:= \S- \A$ is hypoelliptic, since the vector fields appearing in $\mathcal L$ satisfy Hörmander's bracket condition: the missing direction $X_3$ is obtained through the commutator 
$[X_2, X_1]$.

\begin{remark} It was shown that the closure of $\L=\S-\A$ on $D$ is maximal dissipative \cite{GrotStil}. Here invariance with respect to $\mu_\Phi$ means $(\mathcal L f,1)_{\mathcal H}=0$ for all $f\in D$. Since $\Phi$ does not depend on $\theta$, the vector field $X_2=\partial_\theta$ is invariant under $\mu_\Phi$. In contrast, $X_1$ is not invariant with respect to the weighted measure $e^{-\Phi(\xi)}d\xi\otimes d\theta$. The drift term involving $\nabla_\xi\Phi$ in $\mathcal A$ precisely compensates this defect, and ensures invariance of the full generator.
\end{remark}

\noindent that the symmetric operator is given by $\S =  \frac{\sigma^2}{2} \Delta_{\theta}$, its kernel consists of functions that are independent of the angular variable 

\begin{equation}
\ker(\S)
= \{ f \in \Hi : f(\xi,\theta)=g(\xi) \, 
\text{ for some } g:\mathbb R^2 \to \mathbb R \}.
\end{equation}

\begin{remark}
The structure of $\ker(\mathcal S)$ shows the fundamental distinction between coercivity and hypocoercivity.\\

\noindent In the coercive case, the kernel of the symmetric part $\S$ consists only of constant functions and yields exponential convergence to equilibrium. In contrast, in the present hypocoercive setting, the symmetric part acts only on the rotational variable, and its kernel is the larger subspace of functions depending only on $\xi$. \\

\noindent Hence, $\S$ alone is not coercive on the full space. Convergence to equilibrium is obtained only through the interaction with the antisymmetric part $\A$, which transfers dissipation from the rotational direction to the translational variables via the underlying Lie group structure, encoded in the commutator relations between the invariant vector fields $X_1, X_2$ and $X_3$. 
\end{remark}

\begin{definition}[Projection map]
We denote by $\Pi_{\S}: \Hi \to \ker(\S)$ the orthogonal
projection onto this kernel, defined by averaging over the rotational variable

\begin{equation}\label{Orthogonal_Projection}
\Pi_{\S} f (\xi) := \frac{1}{2\pi} \int_{\mathbb{S}^1} f(\xi, \theta) \,  d \theta, \quad (\xi,\theta)\in\mathbb R^2\times \mathbb{S}^1. 
\end{equation}
\end{definition}

\noindent We also introduce the following orthogonal projection on $\Hi$, which removes the constant components, as follows 

\begin{equation}
\Pi := \Pi_{\mathcal S} - (\cdot,1)_{\mathcal H},
\end{equation}

\noindent This projection above is essential for the Poincaré-type inequalities needed in the Hypocoercivity conditions, and is used in the hypocoercive estimates.\\

%%%%%%%%%%%%%%%%%%%%%%%%%%%%%%%%%%%%%%%%%%%%%%%%%%%%%%%%%%%%%%%%%%%%%%%%

\noindent The aim of the following section is to illustrate how to verify the assumptions required by the abstract Hilbert space hypocoercivity framework developed in \cite{GrotStil}, and to establish hypocoercivity for the operator $\L$.

\subsection{Hypocoercivity conditions}\label{subsec:projection}

The operator $\L = \S - \A$ is degenerate: $\S$ acts only on the angular variable $\theta$, so $\ker(\S)$ consists of $\theta$-independent functions. Convergence to equilibrium occurs through coupling between $\S$ and $\A$ via the Lie algebra structure of $\mathfrak{se}(2)$.\\

\noindent We analyze the compositions $\A\Pi$, $\A^2\Pi$, 
$\Pi \A \Pi$, and $\Pi \mathcal A^2\Pi$ to make the coupling explicit. 
The operator $\A$ maps $\ker(\S)$ into 
$\theta$-dependent directions via $X_1$, while $\mathcal A^2$ exploits 
the commutator $[X_2,X_1]=X_3$ to access the remaining translational direction. 
The projected operators $\Pi \mathcal A\Pi$ and $\Pi \A^2\Pi$ 
describe the induced action of transport on $\ker(\S)$. The following propositions determine these operators explicitly, and can be found in \cite{GrotStil}.

\begin{proposition}\label{eq:AP_SE(2)}
Let $f\in D$, and set $g:=\Pi f$. Then

\begin{equation}\label{2.21_SE(2)}
\begin{aligned}
 \A \P f   =X_1 g. 
\end{aligned}
\end{equation}
\end{proposition}

%%%%%%%%%%%%%%%%%%%%%%%%%%%%%%%%

\begin{proof}
Observe that for $f \in D$, the projection $\Pi f = g$ is independent of the rotational variable $\theta$. Applying the operator $\mathcal A$, we obtain

\begin{align}\label{AP}
\mathcal A \Pi f(\xi,\theta)
&=  \cos(\theta)\, \partial_{\xi_1} g(\xi)
   + \sin(\theta)\, \partial_{\xi_2} g(\xi)
   - (\nabla_{\xi} \Phi(\xi)\!\cdot\! v_{\perp}(\theta))\, X_2 g(\xi) \notag\\
&= \cos(\theta)\, \partial_{\xi_1} g(\xi)
   + \sin(\theta)\, \partial_{\xi_2} g(\xi),
\end{align}

\noindent since $X_2 g = \partial_\theta g = 0$.
\end{proof}

\noindent Thus, $\A$ moves functions from $\ker(\S)$ into $\theta$-dependent directions. Although $\S$ alone cannot control functions independent of $\theta$, the operator 
$\A$ introduces angular oscillations through the invariant vector field $X_1$. This reflects the geometric coupling inherent in the Lie group structure and constitutes the first step in the hypocoercive mechanism.

%%%%%%%%%%%%%%%%%%%%%%%%%%%%%%%%

\begin{remark}
Since the integrals of $\sin$ and $\cos$ over $(0,2\pi]$ vanish, we obtain
\begin{equation}
\Pi \A \Pi =0. 
\end{equation}
This cancellation reflects the symmetry of the rotational variable: the orthogonal projection $\Pi$ averages over the compact subgroup 
$\mathbb{S}^1$, so that first-order transport $\A$ effects vanish.
\end{remark}

%%%%%%%%%%%%%%%%%%%%%%%%%%%%%%%%%%%%%%%%%
\begin{proposition}
Let $f\in D,$ and set $g:= \Pi f$. Then

\begin{equation}\label{2.22_SE(2)}
\begin{aligned}
\A^2 \P f %& = \mathcal{A}\left(X_1 g\right)=X_1\left(X_1 g\right)-\left(\nabla_{\xi} \Phi \cdot v_{\perp}\right) X_2\left(X_1 g\right),\\
= & \cos^2 (\theta) \partial_{\xi^1 \xi^1} g+2 \sin(\theta) \cos(\theta) \partial_{\xi^1 \xi^2} g \\[4pt]
& +\sin^2( \theta) \partial_{\xi^2 \xi^2} g -  \left(\nabla_{\xi} \Phi(\xi) \cdot v_{\perp}(\theta)\right) X_3 g(\xi).
\end{aligned}
\end{equation}

\end{proposition}

\begin{proof}
Since $g$ is independent of $\theta$, we have $X_2 g=\partial_\theta g=0$, and hence
$\mathcal A\Pi f = X_1 g$. Applying $\mathcal A$ once more yields
\begin{equation}\label{eq:AAP_split}
\mathcal A^2\Pi f
= \mathcal A(X_1 g)
= X_1(X_1 g) - \big(\nabla_\xi\Phi(\xi)\cdot v_\perp(\theta)\big)\,X_2(X_1 g).
\end{equation}

\noindent For the first term, using $X_1 = v(\theta)\cdot\nabla_\xi$, and that $v(\theta)$ depends only on $\theta$,

\begin{equation*}
\begin{aligned}
X_1(X_1 g)
&= (v(\theta)\cdot\nabla_\xi)\big(v(\theta)\cdot\nabla_\xi g\big)\\
&= \cos^2(\theta)\,\partial_{\xi^1\xi^1} g
  +2\sin(\theta)\cos(\theta)\,\partial_{\xi^1\xi^2} g
  +\sin^2(\theta)\,\partial_{\xi^2\xi^2} g.    
\end{aligned}
\end{equation*}

\noindent For the second term, since $X_2=\partial_\theta$ and $X_2 g=0$, we have

\begin{equation*}
X_2(X_1 g)
= \partial_\theta\big(v(\theta)\cdot\nabla_\xi g\big)
= (\partial_\theta v(\theta))\cdot\nabla_\xi g
= v_\perp(\theta)\cdot\nabla_\xi g
= X_3 g.
\end{equation*}

\noindent Equivalently, using $X_2 g=0$,
\begin{equation*}
X_2(X_1 g) = [X_2,X_1]g = X_3 g.
\end{equation*}

\noindent Combining the two parts proves \eqref{2.22_SE(2)}.
\end{proof}

\noindent The appearance of $X_3$ reflects the geometry of $\SE(2)$: since $X_2=\partial_\theta$ generates rotations and $X_1$ depends on the orientation $\theta$, differentiating $X_1$ in the rotational direction produces the orthogonal translational direction $X_3=[X_2,X_1]$. The following proposition can also be found in \cite{GrotStil}

\begin{proposition} The symmetric operator $G:=\Pi \A^2 \Pi$ on $D$ is given by

\begin{equation}
\begin{aligned}
G:= \Pi \A^2 \Pi&=  \frac{1}{2}\Delta_{\xi}- \frac{1}{2}\nabla_{\xi}\Phi(\xi)  \cdot \nabla_{\xi}. 
\end{aligned}
\end{equation}

\end{proposition}

\begin{proof}
Let $f\in D$, and set $g:= \Pi f$. Note that $\Pi f$ is mean-zero, i.e.
$(\Pi f,1)_{\Hi}=0$, and that $\Pi$ differs from $\Pi_\S$ only by the
projection onto constants. Hence, in the computation below, applying $\Pi$ coincides with averaging over $\theta$ for the expressions considered below, i.e. $\Pi h = \Pi_S h$ for all expressions $h$ arising from $\A^2\Pi f$.

\begin{align}
\Pi \mathcal A^2 \Pi f
&= \frac{1}{2}\,\partial_{\xi^1\xi^1} g(\xi)
   - \frac{1}{2}\,(\partial_{\xi^1}\Phi(\xi))(\partial_{\xi^1} g(\xi)) \notag\\
&\quad + \frac{1}{2}\,\partial_{\xi^2\xi^2} g(\xi)
   - \frac{1}{2}\,(\partial_{\xi^2}\Phi(\xi))(\partial_{\xi^2} g(\xi)),
\end{align}
where we used that
$$
\frac{1}{2\pi}\int_0^{2\pi}\cos^2(\theta)\,d\theta
=\frac{1}{2\pi}\int_0^{2\pi}\sin^2(\theta)\,d\theta=\frac12,
\qquad
\frac{1}{2\pi}\int_0^{2\pi}\sin(\theta)\cos(\theta)\,d\theta=0.
$$
Putting together these terms gives
$$
\Pi \mathcal A^2 \Pi f
= \frac{1}{2}\Delta_{\xi} g(\xi)
  - \frac{1}{2}\nabla_{\xi}\Phi(\xi)\cdot\nabla_{\xi} g(\xi),
$$
which proves the claim.
\end{proof}

\begin{remark}
In particular, averaging over the rotational subgroup yields an elliptic operator in the translation variables; that is, we can observe the Laplacian operator in the variable $\xi$. 
\end{remark}

%%%%%%%%%%%%%%%%%%%%%%%%%%%%%%%%%%%%%%%%%%%%%%%%%%%%%%%%%%%%%%%%%
\subsection{Hypocoercivity assumptions}\label{subsec:hypocoercivity-conditions}

Under the following assumptions in \cite{GrotStil}, the abstract hypocoercivity assumptions were shown:

\begin{itemize}\label{Hypocoercivity_Conditions}

\item[(\textbf{C1})]\label{cond:C1_SE(2)} \,\,\,\, The potential $\Phi: \mathbb{R}^2 \rightarrow \mathbb{R}$ is bounded from below, $\Phi \in C^{2}(\mathbb{R}^2)$ and $e^{-\Phi} d\xi$ is a probability measure on $(\mathbb{R}^2, \mathcal{B}(\mathbb{R}^2))$. \\
    
\item[(\textbf{C2})]\label{cond:C2_SE(2)} \,\,\,\,  The probability measure $e^{-\Phi} d\xi$ satisfies the Poincaré inequality for $\Lambda >0$

\begin{equation}
\|\nabla_{\xi} f\|_{L^2\left(e^{-\Phi} \mathrm{d} \xi \right)}^2 \geq \Lambda\left\|f-(f, 1)_{L^2\left(e^{-\Phi} \mathrm{d} \xi \right)}\right\|_{L^2\left(e^{-\Phi} \mathrm{d} \xi \right)}^2.
\end{equation}

\item[(\textbf{C3})]\label{cond:C3_SE(2)} \,\,\,\ There exists a constant $c< \infty$ such that

\begin{equation*}
|\Delta \Phi(\xi)| \leq c (1+ |\nabla \Phi(\xi)|) \quad \text{for all} \quad \xi \in \mathbb{R}^2.    
\end{equation*}

\end{itemize}

For the convenience of the reader, here we recall the main ideas.

\begin{proposition}[Microscopic Coercivity]
Let $\Phi: \mathbb{R}^2 \rightarrow \mathbb{R}$ a function that satisfies \emph{(C1)} and \emph{(C2)}, and let $\sigma \in (0, \infty)$. Then, the following condition is satisfied, 

\begin{equation}\label{Condition_H2}
-(\S f,f)_{\Hi}
\geq \Lambda_m \|(I-\Pi_{\mathcal S})f\|_{\mathcal H}^2,
\qquad \forall f\in D,
\end{equation}
with $\Lambda_m=\frac12\sigma^2$.
\end{proposition}

\begin{proof}

By the definition of the symmetric part of the operator,

\begin{equation*}
-(\mathcal{S} f, f)_{\mathcal{H}}=-\frac{\sigma^2}{2}\left(\partial_\theta^2 f, f\right)_{\mathcal{H}}=\frac{\sigma^2}{2}\left\|\partial_\theta f\right\|_{\mathcal{H}}^2.
\end{equation*}

Since $\Pi_{\S}$ is the orthogonal projection onto functions independent
of $\theta$, the Poincaré inequality on $\mathbb{S}^1$
(see \cite{beckner1989generalized}) yields

$$
\|\partial_\theta f\|_{L^2(\mathbb S^1)}^2
\ge \|f-\Pi_{\mathcal S}f\|_{L^2(\mathbb S^1)}^2.
$$
Integrating with respect to $\xi\in\mathbb R^2$ proves
\eqref{Condition_H2} with $\Lambda_m=\frac12\sigma^2$.
\end{proof}

\begin{remark}
The constant $1$ reflects the spectral gap of the Laplacian on the circle
$\mathbb{S}^1$.
\end{remark}

%%%%%%%%%%%%%%%%%%%%%%%%%%%%%%%%%%%%%%%%%%%%%%%%%%%%%%%%%%%%%%%%%

\begin{proposition}[Macroscopic coercivity] Let $\Phi:\mathbb{R}^2 \rightarrow \mathbb{R}$ satisfies \emph{(C1)} and \emph{(C2)}. Then the following condition holds

\begin{equation}
\|\A \Pi f\|^2 \geq \Lambda_M\|\Pi f\|^2 \quad \text { for all } f \in D\left(\A \P\right) .
\end{equation}

for $\Lambda_M= \frac{\Lambda}{2}$. 
\end{proposition}

\begin{proof}
Let $f\in D$, and set $g:= \Pi f$. From Proposition \ref{eq:AP_SE(2)}, 

\begin{align*}
\|\A\Pi f\|_{\Hi}^2
&= \frac{1}{2\pi Z}\int_{\mathbb R^2\times\mathbb S^1}
\big(\cos\theta\,\partial_{\xi^1} g
+ \sin\theta\,\partial_{\xi^2} g\big)^2
\,e^{-\Phi(\xi)}\,d\xi\,d\theta \\
&= \frac{1}{2Z}\int_{\mathbb R^2}
|\nabla_\xi g(\xi)|^2\,e^{-\Phi(\xi)}\,d\xi,
\end{align*}

where we used

\begin{equation*}
\frac{1}{2\pi}\int_0^{2\pi}\cos^2\theta\,d\theta
=\frac{1}{2\pi}\int_0^{2\pi}\sin^2\theta\,d\theta=\frac12,
\qquad
\frac{1}{2\pi}\int_0^{2\pi}\sin\theta\cos\theta\,d\theta=0.
\end{equation*}

\noindent By assumption (C2), the probability measure
$Z^{-1}e^{-\Phi(\xi)}\,d\xi$ satisfies a Poincaré inequality with constant
$\Lambda>0$, so that
\begin{equation*}
\int_{\mathbb R^2}|\nabla_\xi g|^2 e^{-\Phi}\,d\xi
\ge \Lambda
\int_{\mathbb R^2}|g-(g,1)_{L^2(e^{-\Phi}d\xi)}|^2 e^{-\Phi}\,d\xi.    
\end{equation*}
Since $\Pi f=g$ and $(g,1)_{L^2(e^{-\Phi}d\xi)}=(f,1)_{\Hi}$, this yields
\begin{equation*}
\|\A\Pi f\|_{\Hi}^2
\geq \frac{\Lambda}{2}\|\Pi f\|_{\Hi}^2,    
\end{equation*}

\noindent which proves the claim, because $D$ is a core of $(\A \Pi,  D(\A \Pi))$.
\end{proof}

%%%%%%%%%%%%%%%%%%%%%%%%%%%%%%%%%%%%%%%%%%%%
\noindent The remaining assumption to prove the hypocoercivity theorem consists of elliptic a priori estimates as developed in \cite{dolbeault2015hypocoercivity}, involving the operator $$B:= (I + (\A\Pi)^{*} (\A\Pi))^{-1}(\A\Pi)^{*}.$$ These bounds ensure that dissipation from $\S$ and transport from $\A$ jointly control the $L^2$ norm. The proof relies on elliptic regularity for $G = \Pi \A^2 \Pi$, and all details of the corresponding estimates in this setting can be found in \cite{GrotStil}.\\

\begin{proposition}[Elliptic a priori estimates]
Assume that $\Phi: \mathbb{R}^2 \rightarrow \mathbb{R}$ satisfies \emph{(C1)}, \emph{(C2)} and \emph{(C3)}, $\sigma \in (0, \infty)$, and $f\in D$. Then for all
$f\in D$, the following estimates hold:

\begin{equation*}
    \begin{aligned}
        \|BSf\| \leq c_1 \|(I-\Pi_{\#})f\|, \qquad \|BA(I-\P)f\| \leq c_2 \|(I-\Pi_{\#})f\|, 
    \end{aligned}
\end{equation*}

\noindent where $\Pi_{\#}$ is either $\P_{\S}$ or $\P$, and the constanst are given by $c_1 = \frac{\sigma^2}{4}$, and $c_{2} \in (0, \infty)$.

\end{proposition}

%%%%%%%%%%%%%%%%%%%%%%%%%%%%%%%%%%%%%%%%%%%%%%%
\noindent We are now in a position to state the hypocoercivity theorem, 
formulated in the Kolmogorov backward setting in \cite{GrotStil}. For the general concept, mainly applied in the Fokker-Planck setting, we refer to \cite{dolbeault2015hypocoercivity}.

\begin{theorem}
Assume that the operator $(\L=\S-\A,D)$ satisfies the structural assumptions and coercivity conditions of the abstract Hilbert space hypocoercivity framework. Then there exist constants $\kappa_1, \kappa_2 \in (0, \infty)$ which depend on $\Lambda_m, \Lambda_M, c_1$ and $c_2$ such that for all $g \in \Hi$, the associated semigroup $(T_t)_{t\ge0}$ satisfies

\begin{equation*}\label{eq:exponential_convergence_to_equilibrium}
    \| T_t g - (g,1)_\Hi\| \leq \kappa_1 e^{-\kappa_2 t} \|g - (g,1)_\Hi\| \quad \text{for all } \quad t \geq 0. 
\end{equation*}
\end{theorem}

\begin{remark}
The constants $\kappa_1$ and $\kappa_2$ inherit their dependence from the microscopic and macroscopic coercivity constants $\Lambda_m$ and $\Lambda_M$, as well as from the elliptic bounds $c_1$ and $c_2$. In the present $\SE(2)$ setting, this reflects the geometry of the state space: the spectral gap of the Laplacian on the compact rotational component $\mathbb{S}^1$ enters through $\Lambda_m$, while the confining properties of the potential $\Phi$ determine $\Lambda_M$ via the macroscopic diffusion on $\mathbb R^2$.
\end{remark}

%%%%%%%%%%%%%%%5
%\vspace{-0.6cm}
\section{Conclusion and outlook}

Although the corresponding convergence properties are previously known results for the corresponding dynamics on $\mathbb{R}^2 \times \mathbb{S}^1$, the present approach highlights how these properties emerge from the intrinsic geometry of $\SE(2)$. This viewpoint clarifies the role of compact and non-compact directions in hypocoercivity, hence provides a natural framework for extensions to other Lie groups.

%-------------------------------------------------------
%\newpage
\addcontentsline{toc}{section}{References}

\bibliographystyle{plain}

%-----------------------------------------------------------------------------------------------------------------------------------------------

\begin{acknowledgements}
\noindent I would like to express my deep gratitude to my advisor, Martin Grothaus, for his patience and for the time he generously took to introduce me to this beautiful problem, allowing me to explore hypocoercivity from a geometric perspective.
\end{acknowledgements}

\vspace{1cm}

\AuthorInfo{Martin Grothaus}{Department of Mathematics\\RPTU University  Kaiserslautern--Landau, Germany}{grothaus@rptu.de}

\vspace{-0.5cm}

\AuthorInfo{Andrea Vanessa Hurtado Quiceno}{Department of Mathematics\\RPTU University  Kaiserslautern--Landau, Germany}{avanessa.hquiceno@edu.rptu.de}
\end{document}